\documentclass[12pt,bezier]{article}
\usepackage{amsmath,amssymb,amsfonts,euscript,graphicx}

\title{\bf Indecomposable Decomposition  and Couniserial Dimension 
}
\author{ {\bf A. Ghorbani}$^{{\rm a}}$, {\bf S. K. Jain$^{{\rm b, c}}$ \thanks { Corresponding author;} and }
{\bf Z. Nazemian$^{{\rm a}}$ \thanks {This paper was presented at the OSU-Denison conference, May 10, 2014 and at the University of California, San Diego, June 21, 2014.}}\\
{\footnotesize{ $^{\rm a}$ Department of Mathematical
 Sciences, Isfahan
University of Technology}}\vspace{-1mm}\\ {\footnotesize{ P.O.Box: 84156-83111, Isfahan, Iran}}\\
{\footnotesize{ $^{\rm b}$ Department of Mathematics, Ohio University, Athens, OH 45701, USA }}\\
{\footnotesize{ $^{\rm c}$ King Abdulaziz University, Jeddah, SA }}\\
  {\footnotesize{a$_{-}$ghorbani@cc.iut.ac.ir}}\hspace{5mm}
 {\footnotesize{jain@ohio.edu}}\hspace{5mm}\\
{\footnotesize{z.nazemian@math.iut.ac.ir  }}\vspace{-1mm}
}

\def\be{\begin{enumerate}}
\def\ee{\end{enumerate}}

\newtheorem{theorem}{Theorem}[section]
\newtheorem{lemma}[theorem]{Lemma}
\newtheorem{proposition}[theorem]{Proposition}
\newtheorem{corollary}[theorem]{Corollary}
\newtheorem{definition}[theorem]{Definition}
\newtheorem{example}[theorem]{Example}

\newtheorem{remark}[theorem]{Remark}

\newenvironment{pproof}{\noindent{\bf Proof. }}{}

\date{}
\input{epsf}
\begin{document}
\maketitle
\begin{abstract}
{\noindent Dimensions like Gelfand, Krull, Goldie have an intrinsic role in
the study of theory of rings and modules.  They  provide useful technical
tools for studying their structure.  In this paper we define one of the
dimensions called couniserial dimension
 that measures how close a ring or module is to being uniform. Despite their  different objectives,  it
turns out that there are certain common properties   between the couniserial 
dimension and Krull dimension like each module having such a dimension
contains a  uniform submodule and has finite uniform dimension, among others.
Like all dimensions, this is an ordinal valued
invariant. Every module of finite length has couniserial dimension and
its value lies between the uniform dimension and the length of the module.
Modules with countable couniserial dimension are shown to possess
indecomposable decomposition. In particular, von Neumann regular ring with
countable couniserial dimension is semisimple artinian. If the maximal right quotient ring of a non-singular ring 
$R$ has a couniserial dimension as an $R$-module, then $R$ is a semiprime right Goldie ring. As one of the applications, it follows 
that  all 
right $R$-modules have couniserial dimension if and only if $R$ is a
semisimple artinian ring. }
\end{abstract}
\noindent{\bf 0. Introduction}
{  In this article we introduce a notion of dimension of a module, to
be called couniserial dimension. It is an ordinal valued invariant that is
in some sense a measure of how far a module is from being uniform. 
In order to define couniserial dimension for modules over a ring $R$, we
first define, by transfinite induction, classes $\zeta _{\alpha }$ of $R$%
-modules for all ordinals $\alpha \geq 1$. First we remark 
that if a module  $M$ is isomorphic to all its non-zero submodules, then $M$ must be uniform.
To start with, let $\zeta _{1}$
be the class of all uniform modules.  Next, consider an ordinal $\alpha >1$; if $%
\zeta _{\beta }$ has been defined for all ordinals $\beta <\alpha $, let $%
\zeta _{\alpha }$ be the class of those $R$-modules $M$ such that for every
non-zero submodule $N$ of $M$, where $N\ncong M$,  we have $N\in \bigcup_{\beta
<\alpha }\zeta _{\beta }$. If an $R$-module $M$ belongs to some $\zeta
_{\alpha }$, then the least such $\alpha $ is called the {\it couniserial dimension}  of 
$M$, denoted by c.u.dim$(M)$. For $M=0$, we define
c.u.dim$(M)=0$. If a non-zero module $M$ does not belong to any $\zeta_{\alpha }$, then we say that  c.u.dim$(M)$ is not
defined, or that  $M$ has no couniserial dimension. Equivalently, Proposition \ref{2.2}
shows that  }{an $R$-module $M$ has couniserial dimension if and only if for each
descending chain of submodules of $M$, $M_{1}\geq M_{2}\geq ...$, there
exists $n\geq 1$, either $M_{n}$ is uniform or $M_{n}\cong M_{k}$ for all $%
k\geq n$.}{ It is clear by the definition that every submodule and so every summand of a
module with couniserial dimension has couniserial dimension.  Also note that, for the
integer number $n$,  couniserial dimension of }$\Bbb{Z}^{n}$ is $n$. {An example is given
to show that the direct sum of two modules each with couniserial
dimension (even copies of a module) need not have couniserial dimension. In Section 2, we prove some basic properties of the couniserial
dimension. In Section 3, we prove  our main results.
It is shown in Theorem \ref{decomposation} that a module of countable
(finite or infinite) couniserial dimension can be decomposed in to
indecomposable modules. Theorem \ref{dedekind} shows that a Dedekind finite module with couniserial dimension is a finite direct 
sum of indecomposable modules. Theorem \ref{2.5} in Section 3 shows that for a right  non-singuar ring  $R$
 with  maximal right quotient ring $Q$, if $Q_R$ has
 couniserial dimension, then $R$ is a  semiprime right Goldie ring which is a finite product of piecewise  domains.
 The reader may compare this  with the wellknown result that a prime ring with Krull dimension is a right Goldie ring
 but need not be a piecewise domain. Furthermore, a prime right Goldie ring need not have couniserial  dimension as is also
 the case for
 Krull dimension. \\ 
\indent In Section 4, we give some applications of couniserial dimension.
It is shown in Proposition \ref{Artinian} that a module $M$ with finite
length is semisimple if and only if for every submodule $N$ of $M$ the right 
$R$-module $\oplus _{i=1}^{\infty }M/N$ has couniserial dimension.  As a
consequence a commutative noetherian ring $R$ is semisimple if and only if
for every finite length module $M$ the module $\oplus _{i=1}^{\infty }M$ has
couniserial dimension.} It is shown in Proposition \ref{anti is injective} that
if $P$ is an anti-coHopfian projective right $R$-module and $\oplus
_{i=1}^{\infty }E(P)$ has couniserial dimension, then $P$ is injective.
 As another application we show that  all right (left) $R$-module have couniserial dimension if and only if $R$ is semisimple
artinian (see Theorem \ref{final}).    
Several examples are included in the paper that demonstrates as to why 
the conditions imposed are necessary and what, if any, there is any relation with the corresponding result in the literature.

\section{\hspace{-6mm}.  Definitions and Notation.}
 Recall that a semisimple module $M$ is said to be {\it homogeneous} if $M$
is a direct sum of pairwise isomorphic simple submodules.
A module $M$ has
 {\it finite uniform dimension} (or {\it finite Goldie rank}) if $M$ contains no infinite direct
sum of non-zero submodules, or equivalently, there exist independent
uniform submodules $U_1, ... ,U_n$ in $M$ such that
$\oplus _{i = 1}^n U_i$ is an essential
submodule of $M$. Note that $n$ is uniquely determined by $M$. In this case, it is written u.dim$(M) = n$.\\
\indent For any module $M$, we define Z$(M) = \lbrace x \in   M : $ r.ann$(x)$ is an essential right ideal of $R \rbrace$ . It can be
easily checked that Z$(M)$ is a submodule of $M$. If Z$(M) = 0$, then $ M$
is called a {\it non-singular} module. In particular, if we take $M = R_R$, then $R$ is called right non-singular
if Z$(R_R) = 0$.\\
\indent A ring $R$ is a called {\it right Goldie} ring if it  satisfies the following two conditions:
(i) $R$ has ascending chain condition on right annihilator ideals and, (ii) u.dim$(R_R)$ is finite. \\
\indent Recall that a ring $R$ is {\it right V-ring} if all right simple $R$-modules are injective.  A ring $R$ is called {\it fully right idempotent}
 if $I = I^2$, for every right ideal $I$. We
recall that a right V-ring is fully right idempotent (see \cite [Corollary 2.2] {7}) and a prime fully right idempotent ring is right
non-singular
(see \cite[Lemma 4.3]{2}). So a  prime right V-ring is right non-singular.
  Recall that a module
$M$ is called {\it $\Sigma$-injective} if every direct sum of copies of $M$ is injective.
A ring $R$ is called {\it  right $\Sigma$-V-ring} if each simple right module is $\Sigma$-injective. \\
\indent In this paper, for a ring $R$, $Q = Q _{max} (R)$ stands  maximal right quotient ring  $R$.
It is well known  that if $R$ is a right  non-singular, then the injective hull of $R_R$,  $E(R_R)$,  is a ring and is equal to 
the  maximal right  quotient ring of $R$, \cite[Corollary 2.31]{Gooderlnonsingular}.\\
\indent A module  $M$ is called {\it Hopfian}  if $M$ is not isomorphic  to  any of its proper factor modules (equivalently, every onto 
endomorphism of 
$M$ is 1-1). 
 {\it Anti-Hopfian} modules are introduced by Hirano and Mogami \cite{Hirano}. Such    modules  are isomorphic to all  its  non-zero
 factor modules. A module $M$ is called uniserial if the lattice of submodules are linearly ordered.  Anti-Hopfian modules are 
 uniserial artinian. \\
  \indent Recall that a module $M$ is called {\it coHopfian} if it is not isomorphic to a proper submodule (equivalently, every 1-1 endomorphism of $M$ is onto).
 Varadarjan \cite{varadarjan} dualized the concept of anti-Hopfian
module and called it  anti-coHopfian module. With slight modification  we will call a non-zero module to be
{\it anti-coHopfian} if is isomorphic to all its non-zero submodules. A non-zero
module $M$ is called {\it uniform} if the intersection of any two non-zero
submodules is non-zero. We see an anti-coHopfian module is noetherian and
uniform.\\
\indent   An $R$-module $M$ has cancellation property if for every $R$-modules $N$ and $T$, $M \oplus N \cong M\oplus T $ implies 
  $N\cong T$.
  Every module with semilocal endomorphism ring has cancellation property  \cite{crash}. Since endomorphism ring of a simple
  module is a division ring,
  it has cancellation property. \\
 \indent Throughout this paper, let $R$ denote an arbitrary ring with
identity and all modules are assumed to be unitary and right modules, unless other words stated. If $N$ is a submodule
(resp. proper submodule) of $M$ we write $N\leq M$ (resp. $N<M$). Also, for
a module $M$, $\oplus _{i=1}^{\infty }M$ stands for countably infinite
direct sum of copies of $M$. If $N$ is a submodule of $M$
  and  $k > 1$, $\oplus _{i = k}^{\infty} N = \oplus _{i = 1}^{\infty} N_{i}$ is a submodule of 
$\oplus _{i = 1}^{\infty} M$
with
$N_1 = N_2 = ... = N_{k - 1} = 0$ and for $i \geq  k$ $N_i = N$. 

\section{\hspace{-6mm}. Basic and Preliminary Results.}

As defined in the introduction,  couniserial dimension is an ordinal valued number. The reader may refer to  \cite {stoll}  regarding 
ordinal 
numbers. We begin this section with a lemma and a remark on the definition of couniserial dimension. 

\begin{lemma} \label{anti-coHopfian}
An anti-coHopfian module  is uniform noetherian.
\end{lemma}
\begin{proof}
Since $M$ is isomorphic to each  cyclic submodules,  $M$ is cyclic and every  submodule of $M$ is cyclic and so $M$ is noetherian.
Thus $M$ has a uniform submodule, say $U$. Since $U \cong M$, $M$ is uniform.$~\square$
\end{proof}

\begin {remark} \label{1.3}
{\rm  We make the convention that a statement  ``c.u.dim$(M) = \alpha$''  will mean that the couniserial dimension of $M$ exists 
and equals $\alpha$. By the definition of couniserial dimension,
 if $M$ has  couniserial dimension and $N$ is a submodule of $M$,
  then $N$ has couniserial dimension and c.u.dim$(N) \leq $ c.u.dim$(M)$.
  Moreover,   if $M$  is not uniform  and c.u.dim$(M) = $ c.u.dim$(N)$, where $N$ is a submodule of $M$,
   then $M \cong N$.
    On the other hand,
since every set of ordinal numbers has supremum, it follows immediately from the definition that  $M$ has 
couniserial dimension  if and only if for all submodules $N$ of $M$  
with  $N \ncong M$ , c.u.dim$(N)$ is  defined.  In  the latter case, 
if  $\alpha = $ sup$\lbrace$ c.u.dim$(N) ~ \vert \   N \leq M,  N \ncong M \rbrace  $, then   c.u.dim$(M) \leq \alpha + 1$. }
\end{remark}

%--------------------------------------------------------------------------------------------------------------------------------------------------------------------

The next proposition provides a working  definition for a module $M$ that has couniserial dimension.

\begin{proposition}\label{2.2}
 An $R$-module $M$ has couniserial dimension if and only if  for every descending chain of submodules
$  M_1 \geq M_2 \geq ... $,  there exists $ n \geq 1$ such  that
$M_n$ is uniform or  $ M_n \cong  M_k$ for all  $k \geq n $.
\end{proposition}
\begin{proof}
{\rm ($\Rightarrow$) Let
$  M_1 \geq M_2 \geq ... $ be a descending chain of submodules of $M$.
Put $\gamma =$ inf $\lbrace$c.u.dim$(M_n) ~ \vert ~ n\geq 1\rbrace  $. So $\gamma = $ c.u.dim $(M_n)$ for some $n \geq 1$.
If $ M_n$ is not uniform,  then $ M_n \cong M_k$ for all $k \geq n $, because $\gamma $ is infimum. \\
($ \Leftarrow $)   If  $M$ does not have couniserial dimension,
then $M$ is not uniform and so there exists a  submodule $M_1$ of $M$ such that $M_1 \ncong M$ and $M_1$ does not
   have couniserial dimension, by the above remark. So
there exists a submodule $M_2$ of $M_1$ such that $M_2 \ncong M_1$ and $M_2$ does not   have couniserial dimension.
Continuing in this manner, we obtain a descending  chain of submodules $   M_1 \geq M_2 \geq ... $, such that
 for every $i \geq 1$,  $M_i$ does not   have couniserial dimension and $M_i \ncong M_{i+1}$, a contradiction. This completes the 
 proof.  $~\square$  }
\end{proof}
%--------------------------------------------------------------------------------------------------------------------------------------------------------------

 As a consequence, we have the following corollary. 

\begin{corollary}\label{2.3}
Every artinian module  has  couniserial dimension.
\end{corollary}
%%%%%%%%%%%%%%%%%%%%%%%%%%%%%%%%%%%%%%%%%%%%%%%%%%%%%

\begin{lemma} \label{less}
If $M$ is an $R$-module and {\rm c.u.dim}$(M) = \alpha$, then for any  $ 0 \leq \beta \leq \alpha$,
 there exists a submodule  $N$ of $M$ such that {\rm c.u.dim}$(N) = \beta$.
 \end{lemma}
 \begin{proof}
 {\rm The proof is by transfinite induction on c.u.dim$(M) = \alpha$. The case $\alpha =  1$ is  clear.
  Let
 $\alpha > 1$ and $0 \leq \beta < \alpha$,
  then, using Remark \ref{1.3}, there exists a submodule $K$ of $M$ such that $K \ncong M$ and
  $\beta \leq$ c.u.dim$(K)$. Now since $\beta \leq$ c.u.dim$(K) < \alpha$, by induction hypothesis,
  there exists a submodule $N$ of $K$ such that c.u.dim$(N) = \beta$. $~\square$  }
 \end{proof}

As a consequence we have the following.

\begin{lemma} \label{uniform submodule}
 Every module with couniserial dimension has a uniform submodule.
 \end{lemma}

 In the next proposition
 we observe that  every module of finite couniserial dimension has finite  uniform dimension.

 \begin{lemma} \label{non}
  Let  $M$ be an $R$-module of finite couniserial dimension.  Then $M$ has finite uniform dimension  and {\rm u.dim}$(M) \leq$ {\rm c.u.dim}$(M)$.  
 \end{lemma}
 \begin{proof}
 The proof is by induction on c.u.dim$(M) = n$.  The case $n = 1$ is clear.  Let  $n > 1 $   and $N$ be a submodule of $M$
 such that c.u.dim$(N) = n - 1 $.   Thus by the inductive hypothesis, $N$ has  finite uniform dimension. Put  $m = $ u.dim$(N)$.
  If $N $ is not essential in $M$,
 then there exists a uniform submodule $U$ of $M$ such that $N \cap U = 0 $. Thus $N \oplus U$ is a submodule of $M$ of uniform
  dimension $m + 1$. Then $ (N \oplus U) \ncong N$ and so $ n - 1 < $ c.u.dim$(N \oplus U) \leq n$.
   Thus  $(N \oplus U) \cong M$, by Remark \ref{1.3}.  This proves the lemma.   $~\square$ 
   \end{proof} 

\begin{example}
{\rm There exist modules of infinite couniserial dimension but of finite uniform dimension. 
 Take $M = \Bbb{Z}_{{p}^{\infty}} \oplus \Bbb{Z}_{{p}^{\infty}}$.  Then $M$ is artinian $\Bbb{Z}$-module
of infinite couniserial dimension but of finite uniform dimension $2$.}   
\end{example}

In the following we consider equality  in the above lemma in a special case.
 
  \begin{lemma} \label{non1}
  Let  $M$ be an injective non-uniform  $R$-module of finite couniserial dimension.   Then  
  {\rm c.u.dim}$(M) = $ {\rm u.dim}$(M)$
   if and only if $M$ is finitely generated semisimple module.
 \end{lemma}
 \begin{proof}{ $(\Leftarrow) $  is clear.\\
 $(\Rightarrow)$. Let   c.u.dim$(M) = $ u.dim$(M) = m > 1$.
 Then  $M = E_1 \oplus ... \oplus E_m$, where  $E_i$ are  uniform injective modules.
  If $E_1$ is not simple then there exists a non-injective
 submodule
  $K$ of $E_1$. Thus  $K \oplus E_2 \oplus ... \oplus E_m$ is not isomorphic to $M$. But clearly
  c.u.dim$(K \oplus E_2 \oplus ... \oplus E_m) \geq m$, a
   contradiction. This completes the proof. 
   $~\square$}
 \end{proof}

Note that the condition being injective is necessary in the above proposition. 

\begin{example}
{\rm  We can see easily that for $M = \Bbb{Z} \oplus \Bbb{Z}$, 
c.u.dim$(M) = $ u.dim$(M)$ $ = 2$ but $M$ is not semisimple. Also,  the next lemma  shows that there exists a module of finite uniform dimension  without couniserial dimension.}
 \end{example}

The following lemma shows that direct sum of two uniform modules may not have couniserial dimension.

\begin{lemma} \label{example}
Let   $D$ be  a domain and $S $ be a simple $D$-module. If   $S \oplus D$ as $D$-module  has couniserial dimension, then $D$ is principal
right  ideal 
domain. 
\end{lemma}
  \begin{proof}
   Let $I$ be a non-cyclic right ideal of $D$.   Choose a non-zero element $x \in I$. Set  $J_1 = xR$ which is isomorphic to $D$.
Thus
  there exists a right  ideal $J_2$ of $D$ such that $J_2 \cong I$ and $J_2 \leq J_1$. Now  let $J_3$ be
 a cyclic right ideal contained in $J_2$ and by continuing this manner we have a descending chain
 $J_1 \geq J_2 \geq ...$ of right ideals of $D$
  where for each odd integer $i$, $J_i$ is cyclic and for each even integer $i$, $J_i$ is not cyclic. Now consider
 the descending chain $S \oplus J_1 \geq S \oplus  J_2 \geq ...$ of submodules of $S \oplus D$. Since $S$ has cancellation  property
 and for each $i$, $S \oplus J_i$ is not uniform, by using Proposition \ref{2.2}, we see that, for some $n$, 
 $J_n \cong J_{n + 1}$, a contradiction. Thus $D$ is a principal right ideal domain.
\end{proof}  

\begin{remark}
 {\rm (1)
 The simple module $S$ in the statement of the Lemma \ref{example}
  can be replaced by any cancellable module. Indeed it follows from the 
 Theorem \ref{2.5},  proved latter   if the  maximal right quotient ring $Q$ of a domain
 $D$  as $D$-module has couniserial dimension,
 then  $Q_D$ has cancellation property and so if $Q \oplus D$ as $D$-module has couniserial dimension, 
  $D$ must  be right principal ideal domain. \\
 (2)  Also, since a Dedekind domain has cancellation property,   similar proof  shows
  that if $D$ is a Dedekind domain which  is 
 not right principal ideal domain, then $D \oplus D$ does not have couniserial dimension. This example shows that even 
 direct sum of a uniform module with itself may not have couniserial dimension.
 }
\end{remark}

The definition of addition two ordinal numbers can be given inductively.
   If $ \alpha $ and $ \beta $  are two ordinal numbers
  then $ \alpha + 0 = \alpha $,  $ \alpha + (\beta + 1) = (\alpha + \beta) + 1 $ and if $ \gamma $ is a limit ordinal then
   $ \alpha+\gamma $ is the limit of $ \alpha + \beta $ for all $ \beta < \gamma $ (See \cite{stoll}).

  \begin{lemma} \label{ord}\rm (See \cite [Theorem 7.10]{stoll}).   For ordinal numbers  $\alpha $,  $\beta $ and $\gamma$,
  we have the following:\\
{\rm  (1)} If  $\alpha < \beta $, then $\gamma + \alpha < \gamma + \beta $. \\
  {\rm (2)} If $\alpha < \beta $, then $ \alpha + \gamma \leq \beta + \gamma$.
  \end{lemma}

We call an $R$-module $M$ {\it fully coHopfian} if every submodule of $M$ is coHopfian. Note that
 artinian modules are fully coHopfian.
If $I$ is  the set of prime numbers, then  $\oplus _{p \in I} {\Bbb {Z} _p}$ is an  example of fully coHopfian
$\Bbb{Z}$-module that it is not artinian.
   
\begin{proposition}{\label{b}} Let $M = M_1 \oplus M_2$ be a fully coHopfian $R$-module with couniserial dimension.
  Then c.u.dim$(M) \geq $ c.u.dim$(M_1) +  $ c.u.dim$(M_2)$.
  \end{proposition}

    \begin{proof}
   {\rm   We may assume  $M_1, M_2 \neq 0$. We use transfinite   induction on  c.u.dim$(M_2) = \alpha$.
     Since  $M_1 \ncong M$,
   c.u.dim$(M)$ $\geq $ c.u.dim$(M_{1}) + 1$. So the case $\alpha = 1$ is clear.  Thus, suppose  $\alpha > 1$
   and for every right $R$-module $L$ of couniserial dimension less than $\alpha$, c.u.dim$(M_1 \oplus L ) \geq  $ c.u.dim$(M_{1}) + $
   c.u.dim$(L)$.
    If $\alpha$ is a successor ordinal, then there exists an ordinal number $\gamma$ such that
     c.u.dim$(M_2) = \gamma  + 1$. Using Lemma \ref{less}, there exists a non-zero submodule $K$ of $M_2$
   such  that c.u.dim$(K) = \gamma < \alpha $.  So by induction hypothesis $${\rm c.u.dim}(M_1) +\gamma  =
   {\rm c.u.dim}(M_1) +  {\rm c.u.dim}(K) \leq  {\rm c.u.dim}(M_1 \oplus K).$$
    Using our assumption and Remark \ref{1.3}, we have  c.u.dim$(M_1 \oplus K) <  $ c.u.dim$(M)$ and hence
       c.u.dim$(M_{1}) + $  c.u.dim$(M_{2}) \leq$ c.u.dim$(M)$.\\
    \indent  If $\alpha$
     is a limit ordinal and $ 1 \leq \beta < \alpha$,
   then by Remark \ref{1.3}, there exists a  non-zero submodule $K$ of $M_2$ such that
    $\beta \leq $ c.u.dim$(K)$.
    Then by induction hypothesis ${\rm c.u.dim}(M_1)  + \beta \leq  {\rm c.u.dim}(M_1) +  {\rm c.u.dim}(K) \leq
    {\rm c.u.dim}(M_1 \oplus K) <  {\rm c.u.dim}(M).$
    Therefore  c.u.dim$(M_1) + \alpha = $ sup$\lbrace$ c.u.dim$(M_{1}) + \beta \mid \beta < \alpha
    \rbrace \leq  $  c.u.dim$(M)$. $~\square$   }
\end{proof}

The condition fully coHopfian  of Proposition \ref{b} is necessary.  
\begin{example}
{\rm For  the $\Bbb{Z}$-modules $M = \oplus_{i = 1}^{\infty} \Bbb{Z}_{p}$,
 and $ L = \Bbb{Z}_{p}$, we have
$M \cong M \oplus  L$. One can see c.u.dim$(M) = \omega$ and so
c.u.dim$(M) \ngeq $ c.u.dim$(M) + $ c.u.dim$(L)$.
Also,  in general, we don't  have the equality
in Proposition \ref{b}.  Consider the $\Bbb{Z}$-module
   $ M = \Bbb{Z}_2  \oplus \Bbb{Z}_4$. Then,  $M$ is fully coHopfian and $3 = $ c.u.dim$(M) > $
   c.u.dim$(\Bbb{Z} _2) + $ c.u.dim$(\Bbb{Z} _4 )$.}
   \end{example}

Here we prove another result on fully coHopfian module:

\begin{proposition} \label{simple2}
Let $M$ be an $R$-module and $N$ be a cancellable module (for example a simple module)
 such that $N \oplus M$ has couniserial dimension. If $M$ is fully coHopfian, then $M$ is artinian.
\end{proposition}
\begin{proof}
 Let $M$ be fully coHopfian and let $M_1 \geq M_2 \geq ... $ be a descending chain of submodules of $M$. Then
$N \oplus M_1 \geq  N \oplus M_2 \geq ... $ is a descending chain of submodules of $N \oplus M$ and so for some $n$,
$M_i \cong M_n$ for each $i \geq n$. Now since $M$ is fully coHopfian, we have $M_i = M_n$ for each $i \geq n$.
 $~\square$  
 \end{proof}

Let us recall the definition of uniserial dimension \cite{j.algebra}.

\begin {definition} \label{uniserial dimension}
{\rm  In order to define uniserial dimension for modules over a ring $R$, we
first define, by transfinite induction, classes $ \zeta_\alpha $ of $R$-modules for all ordinals
$ \alpha \geq 1 $. To start with, let $ \zeta_1 $ be the class  of non-zero  uniserial modules.
Next, consider an ordinal $ \alpha > 1 $; if $ \zeta_\beta $ has been defined for all ordinals $ \beta < \alpha $,
let $ \zeta_\alpha $ be the class of those $R$-modules $M$ such that, for every submodule $N <  M$, where $M/N \ncong M$, we have
 $M/N \in \bigcup_{\beta < \alpha} \zeta_\beta$. If an $R$-module $M$ belongs to some $\zeta_\alpha$,
then the least such $\alpha$ is the
{\it uniserial dimension of} $M$, denoted u.s.dim$(M)$. For $M = 0$, we define u.s.dim$(M) = 0$.
 If $M$ is non-zero and $M$ does not
belong to any $\zeta_\alpha$, then we say that  ``u.s.dim$(M)$ is not defined,'' or that `` $M$ has no  uniserial dimension.''}
\end {definition}

\begin{remark}\label{semisimple eq}
{\rm Note that, in general,
 there is no relation between  the existence of   the uniserial dimension  and   the 
 existence of the couniserial dimension of a module. For example, the polynomial ring in infinite number of commutative 
 indeterminates 
 over a field $k$,  $R = k[x_1, x_2, ...]$ has this property that
  c.u.dim$(R_R) = 1$ but  $R_R$ does not have uniserial dimension (see 
 \cite [  Remark 2.3] {j.a.its}).
It follows by the definition that   a semisimple module $M$ has uniserial dimension if and only if $M$ has couniserial dimension, in 
which case
  u.s.dim$(M) = \alpha$ if and only if c.u.dim$(M) = \alpha$.
Furthermore a semisimple module   $M$ has couniserial dimension if and only if $M$ is a  finite direct sum of
  homogeneous semisimple modules ( see \cite [Proposition 1.18]{j.algebra}).}
\end{remark}

 Using the  above remark we have the following
interesting results.
\begin{corollary} \label{finite simple}
All right semisimple  modules over a ring $R$ have couniserial dimension if and only if
   there exist only finitely many non-isomorphic  simple right $R$-modules.
\end{corollary}

\begin{lemma} \label{semisimple1}
  Suppose that $M$ is simple. Then $\oplus _{i = 1}^{\infty} E(M)$ has couniserial dimension  if and only if
  $M$ is injective.
\end{lemma}
\begin{proof}
  ($\Leftarrow$)  It is clear by the statement in Remark \ref{semisimple eq}. \\
 ($\Rightarrow$)  Consider the
   descending chain
  $$ M \oplus ( \oplus _{ i = 2}^{\infty}  E(M)) \geq  M^{(2)} \oplus ( \oplus _{ i = 3}^{\infty} E(M)) \geq ... $$
   of submodules of
  $\oplus _{i = 1}^{\infty} E(M)$ where $M^{(n)} = \oplus_{i = 1}^{\infty} M_i $ with  $M_1 = ... = M_n = M$
   and for each $i > n$, $M _i = 0$.
   Then, by Proposition \ref{2.2},
   there exists $n\geq 1$ such that
   $$M^{(n)} \oplus ( \oplus _{i = n + 1}^{\infty}E(M))\cong M^{(n + 1)}\oplus(\oplus _{i = n + 2}^{\infty} E(M))$$
    and so
    $\oplus _{i = n + 1}^{\infty} E(M) \cong M \oplus (\oplus _{i = n + 2}^{\infty} E(M))$, because $M$ is cancelable . Since $M$ 
    is cyclic, there exists a  right module $L$ such that  for some $k$, $E(M) ^{k} \cong M \oplus L $. This shows $M$ is injective. $~\square$
\end{proof}

\section{\hspace{-6mm}. Main Results.}

In this section we use our basic results to prove the main results.  

\begin{proposition} \label{simple1}
Let $M_R$ be an injective module  and $N_R$ be a cancellable module over a commmutative ring $R$
  (for example a simple module)
 such that $N \oplus M$ has couniserial dimension. Then $M$ is $\Sigma$-injective.
\end{proposition}
\begin{proof}
 According to  \cite [Theorem 6.17]{cyclic}, it is enough to show that $R$ satisfies  the ascending chain condition on  ideals
  of
 $R$ that are annihilators of subsets of $M$. Let $I_1 \leq I_2 \leq ... $ be a chain of such annihilator ideals.
  Then for each $i$,
$M_i = $ ann$ _{M}(I_i) $
 is a submodule of $M$ and so we have  descending chain $N \oplus M_1 \geq N \oplus M_2 \geq ... $
of submodules of $N \oplus M$.  Then  there exists a positive integer  $n$ such that $M_n \cong M_i$ for all $i \geq n$. Thus
ann$(M_i) = $ ann$(M_n)  $. Therefore for each   $i \geq n$, $ I_i = I_n$. $~\square$
\end{proof}

\begin{remark}
 {\rm One can see that the above result  provides another proof for the fact that commutative
  V-rings (i.e, von Neumann regular rings ) are
$\Sigma$-V-ring. For an example of  right V-rings that is not $\Sigma$-V-ring, the reader may refer to 
\cite [ Example, page 60]{cyclic}.}  
\end{remark}
 
 The next result shows that if a module 
  has countable couniserial dimension then it can be decomposed into indecomposable modules.

 \begin{theorem}\label{decomposation}
   For an $R$-module $M$, if  {\rm c.u.dim}$(M) \leq \omega $, then $M$ has indecomposable decomposition.
  \end{theorem}
  \begin{proof}
  The proof is by induction on c.u.dim$(M) = \alpha$. The case $\alpha = 1$  is clear. If $\alpha >  1$ and
  $M$ is not indecomposable, then $M = N_1 \oplus N_2$, where $N_1$ and $N_2$ are non-zero submodules of $M$.
  If c.u.dim$(N_i) < $ c.u.dim$(M)$,  $ i = 1, 2$, then by induction hypothesis $M$ has  indecomposable decomposition.
  If not, for definiteness let
   c.u.dim$(N_1) = $ c.u.dim$(M)$.  Then $M \cong N_1$, by Remark \ref{1.3}. Thus
    it contains an infinite direct sum of uniform modules, say
   $\oplus_{i = 1}^{\infty} K_i $. Clearly, c.u.dim$(\oplus_{i = 1}^{\infty} K_i ) \geq \omega$.
    Thus we have $M \cong \oplus_{i = 1}^{\infty} K_i $.
   $~\square$
\end{proof}

\begin{remark}
{\rm We do not know whether the above proposition holds for a module of arbitrary  couniserial dimension. 
For infinite countable couniserial dimension one can show under some condition that the module can be represented as a direct sum 
of
 uniform modules. }
\end{remark}

Recall that a module $M$ is called {\it Dedekind finite} if $M$ is not isomorphic to
any proper direct summand of itself. Clearly, every direct summand of a Dedekind finite module is a Dedekind finite module. 
Obviously, a Hopfian module is Dedekind finite. 
 Since  all finitely generated modules over a commutative ring are Hopfian (see \cite{good}),
  they provide  examples of Dedekind finite modules.
   
\begin{theorem} \label{dedekind}
If $M$ is a Dedekind finite module with couniserial dimension, then $M$ has finite indecomposable decomposition. 
\end{theorem}
\begin{proof}
The proof is by induction on c.u.dim$(M) = \alpha$. The case $\alpha = 1$ is clear. Let $\alpha > 1$ and every
 Dedekind finite module with c.u.dim less than $\alpha$ be decomposed to finitely many 
   indecomposable modules. If $M$ is not indecomposable,
 then $M = M_1 \oplus M_2$. Since $M_i \ncong M$, using Remark \ref{1.3}, c.u.dim$(M_i) <  $ c.u.dim$(M)$ and so, 
 by induction hypothesis, $M_i$ have finite indecomposable decomposition. This completes the proof.    $~\square$
\end{proof}

A ring $R$ is called a {\it von Neumann regular ring} if for each $x \in  R$, there exists
$y \in R$ such that $xyx = x$, equivalently, every principal right ideal is a direct summand.  $R$ is 
 {\it unit regular ring} if for each $x \in  R,$ there exists a unit element $u \in R$
such that $x = xux.$  As a consequence of the above theorem we have the following corollary. 

\begin{corollary}
Every Dedekind finite  von Neumann regular ring (in particular, unit regular rings )
 with couniserial dimension is semisimple artinian.
\end{corollary}

A ring $R$ is called a PWD {\it (piecewise domain)} if it
possesses a complete set $ \lbrace e_{i}  \vert 0 \leq  i  \leq n  \rbrace$ of orthogonal idempotents such that $xy = 0$
implies $x = 0 $ or $y = 0$ whenever $x \in  e_i Re_k$ and $y \in  e_k Re_j $. Note that the
definition is left-right symmetric and all  $e_i R e_i$ are domain, see \cite{lance small}. 

An element $x$ of $R$ is called regular if its right and left annihilators are zero.

\begin{proposition} \label{semiprime right Goldie}
Let  $R$ be a semiprime right Goldie ring  with couniserial dimension. If u.dim$(R_R) = n$, then $R$ has a decomposition 
into 
$n$ uniform modules. In particular, it is a piecewise domain.     
\end{proposition}
\begin{proof}
We can assume that $n > 1$. 
Let $I_1 = U_1 \oplus ... \oplus U_n$ be an essential right ideal of $R$.  Then, by  \cite [Proposition 6.13]{goodearl},
$I_1$ contains a regular element $x$ and thus $J_1 = xR$ is a right ideal of $R$ which is $R$-isomorphic to $ R$. 
So u.dim$(J_1) = n$ and it contains an essential  right ideal $I_2$ of $R$ such that it  is a 
 direct sum of $n$ uniform right ideals. 
  By continuing in  this manner we obtain a descending chain $ I_1 \geq J_1 \geq I_2 \geq ...$ of right ideals of 
  $R$ such that $I_i$ are direct sum of 
$n$ uniform and $J_i$ are isomorphic to   $R$. Since $R$ has couniserial dimension, for some $n$, $ I_n \cong R$.  The last 
statement follows from \cite [Pages 2-3]{lance small}. This 
completes the proof. $~\square$
\end{proof}

\begin{remark} \label{example of prime right Goldie}
{\rm 
There  exists an example of simple noetherian ring of uniform dimension $2$ which has no
non-trivial idempotents (c.f. \cite [Example 7.16, page 441 ] {robson}). 
So by the above proposition this provides an   example   of prime right Goldie ring without couniserial dimension.
  }
\end{remark}

\begin{lemma}\label{Q-map}
Let $R$ be a right non-singular ring with    maximal right 
quotient ring $Q$. Let $M$ be  a $Q$-module. If $M$ is non-singular $R$-module,
such that  $M_R$ has couniserial dimension, then 
 $M_Q$ has couniserial dimension.  
\end{lemma}
\begin{proof} 
  {\rm    Let $M \geq M_1 \geq M_2 \geq ...$ be a descending chain of $Q$-submodules 
of $M$. So it is a descending chain of 
$R$-submodules of $M$ and thus, for some $n$, $M_n$ is uniform $R$-module or 
$M_n  \cong M_i$ as $R$-modules for all $i \geq n$. If $M_n $ is uniform $R$-module, then it is also
 uniform $Q$-module. So let  
$M_n \cong M_i$ as $R$-modules and let  
  $\varphi_{i}$ be  this isomorphism. 
   If $q \in Q$ and
$t \in M_n$
there exists an essential right ideal $E$ of $R$ such that $qE \leq R$. 
So $\varphi_{i} (tqE) = \varphi_{i} (tq) E $ and also $ \varphi_{i} (tqE)  = \varphi_{i} (t) qE$. 
Then $\varphi_{i} (tq) E = \varphi_{i} (t) qE$.
 Since  $Q$ is
right non-singular,  $\varphi_{i} (tq) = \varphi_{i} (t)q$.  Thus  $\varphi_{i}$ is a $Q$-isomorphism.
This completes the proof. $~\square$} 
\end{proof}

A ring  $R$ is semiprime (prime) right Goldie ring if and only if its   maximal right quotient ring  is  
 semisimple (simple) artinian ring,  \cite [Theorems 3.35 and 3.36]{Gooderlnonsingular}.
 Semiprime right Goldie rings are non-singular. 
  A right non-singular ring $R$ is semiprime right Goldie ring if and only if u.dim$(R_R)$ is
   finite,  \cite [Theorem 3.17]{Gooderlnonsingular}. Recall that a {\it right full linear ring} is the ring of all linear transformations
(written on the left) of a right vector space over a division ring. If the dimension of the vector space is finite, a right full linear ring is 
exactly a simple artinian ring.

\begin{theorem}\label{2.5}
Let $R$  be a  right non-singular ring with  
maximal right quotient ring, $Q$.  If $Q$ as an  $R$-module  has couniserial dimension,
 then $R$ is a semiprime right Goldie ring which  is a finite product of prime Goldie rings, each of which is 
 a piecewise domain.  
\end{theorem}
 \begin{proof}
 It is enough to show that $R$ has finite uniform dimension.
 Since $Q_R$ has couniserial dimension, $R_R$ has couniserial dimension and so every right ideal of $R$ has 
 couniserial dimension. Thus   Lemma \ref{uniform submodule} implies that every right ideal contains a uniform submodule. 
 Now by  
\cite [Theorem 3.29] {goodearl}  the   maximal right  quotient ring  of
 $R$  is a product of right full linear rings, say $Q = \prod  _{ i \in I} Q_{i}$, where $Q_{i}$ are right full linear rings. Note that 
 since $R_R$  is  right non-singular, $Q_R$ is also  non-singular and so,  using  Lemma \ref{Q-map}, $Q_Q$ has
  couniserial dimension. 
 At first we claim each $Q_{i}$ is endomorphism ring of a finite dimensional vector space. Assume the contrary. Then 
 $Q_{j}$ is  the  endomorphism ring of an infinite dimensional vector space, for some $j$. Thus  $Q_{j} \cong  Q_{j} \times Q_{j} $ and so 
 if $\iota: Q_{j} \longrightarrow Q$ be the canonical embedding, then $\iota ( Q_{j})$ is a right  ideal of $Q$ and  there exists a
  $Q$-isomorphism
 $Q \cong \iota ( Q_{j})  \times Q$.  Then
  there exist right ideals $T_1$ and $T$ of $Q$ such that $Q = T_1 \oplus T$,    $T_1$ and $Q$ are isomorphic  as
 $Q$-modules and $T \cong  \iota ( Q_{j})$ as $Q$-module. Because 
 $Q_{j}$   is the  endomorphism ring of an  infinite dimensional vector space, it has a right ideal which is not principal, for example
 its socle. So $\iota ( Q_{j})$ and thus $T$ contains a non-cyclic right ideal of $Q$ and thus since $T \cong Q/T_1$, 
 there exists a non-cyclic right ideal of $Q$, say $K_1$ such that $ Q \geq K_1 \geq T_1 $.
  Now $T_1 $ is isomorphic to $Q$. So
  we can have a descending chain
$Q >    K_1 > T_1 > K_2 > T_2 > ... $ of right ideals of $Q$ such that $T_i$ are cyclic but $K_i$ are not
 cyclic. This is a contradiction. So all $Q_i$ are 
 endomorphism ring of  finite dimensional vector spaces. Now to show $R$ is semiprime right Goldie ring it is enough to show
 that   
 the index set $I$ is finite. 
 If $I$ is infinite,  there exist infinite subsets $I_1 $ and $I_2$ of $I$ such that $I = I_1 \cup I_2$.
 and $I_1\cap I_2 $ is empty. Let $T_1 = \prod _{i \in I} N_i$ such that $N_i = Q_i$ for all $i \in I_1$ and 
 $N_i = 0$ for all $i \in I_2$. Similarly let
 $ T = \prod _ {i \in I} M_i$ such that $M_i = Q_i$ for all $i \in I_2$ and 
 $M_i = 0$ for all $i \in I_1$.  Then $T_1$ and $T$ are right ideals of $Q$ and $Q = T_1 \oplus T$.
 $T$ contains a right ideal of $Q$ which is not cyclic, for example $\oplus_{i \in I} M_i$. Since $T \cong Q/ T_1$, there exists a
  non-cyclic right ideal $K_1$ of $Q$ such that 
 $Q \geq K_1 \geq T_1$. Note that $T_1$ is a cyclic $Q$-module and because $I_1$ is infinite, the structure of $T_1$ is 
 similar to that of $Q$.  We can continue in this  manner and 
 find a descending chain of right ideals of $Q$ such that $K_i$ are non cyclic $Q$-modules and $T_i$ are cyclic $Q$ modules,
 which is a contradiction.   
Therefore $I$ is finite and $R$ must have finite uniform dimension. This  shows $R$ is semiprime right  Goldie ring and so
Proposition \ref{semiprime right Goldie}
and \cite [ Corollary 3]{lance small} imply that it is a direct sum of prime right Goldie rings. $~\square$

 \end{proof}

 The reader may ask what if $R_R$ has couniserial dimension instead of $Q_R$.  
Indeed we may point out that unlike a semiprime ring with right Krull dimension, a semiprime ring
with couniserial dimension need not be  a right Goldie ring. See Dubrovin \cite{Uniserial with nil} 
that contains an example of a primitive uniserial ring with non-zero nilpotent elements.  \\

Next we show  that the converse of the above theorem is not true, in general. In fact we show that there exists a prime right Goldie ring $R$
such that c.u.dim$(R_R) = 2$ and  $Q _R$ does not have couniserial dimension. We need the following lemma to give the
 example.   

\begin{lemma}\label{morita}
For an ordinal number $\alpha$, being of couniserial dimension $\alpha$ is a Morita invariant property for modules.
\end{lemma}

\begin{proof}
This is clear by  the definition of couniserial dimension  and \cite [ Proposition 21.7 ]{Anderson}. $~\square$
\end{proof}

\begin{example}
{\rm Here we give an example of a prime right Goldie ring $R$ with maximal right 
 quotient  ring $Q$ such that $Q_R$ does not have couniserial dimension. Take $R = M_2 (\Bbb{Z})$, the  $2 \times 2$ matrix ring over $\Bbb{Z}$.
  Then $R$ is a prime right Goldie ring with maximal right quotient ring 
$Q = M_2(\Bbb{Q})$.  Note that under the standard Morita equivalent between the ring 
$\Bbb{Z}$ and  $R= M_2(\Bbb{Z})$, see \cite [Theorem 17.20 ]{Lam}, $R$ corresponds to
 $\Bbb{Z} \oplus \Bbb{Z}$
and so using the above lemma $R$ has couniserial dimension $2$.     
  If $\lbrace p_i \vert i \geq 1 \rbrace$ is the set of all prime numbers,
 then $\Bbb{Q}/ \Bbb{Z} = \sum _{i = 1}^{\infty} K_i /\Bbb{Z}$,
  where $K_i = \lbrace m/p_{i}^{n} \vert n \geq 0 $ and $m \in \Bbb{Z}\rbrace$. Then 
  take $Q_{n} =  \sum _{i = n} ^{\infty} K_i$. Then $M_2(Q_1) \geq M_2({Q_2}) \geq ... $ is a descending chain
   of  $R$-submodules of $Q$  which are not uniform $R$-modules.
     Assume that for some $n$, $M_2({Q_{n}}) \cong M_2({Q_{n + 1}})$  with 
   an  $R$-isomorphism $\phi$. Let $\phi (\left( \begin{array}{ccc} 1& 0 \\ 0& 1 \end{array}\right)) = 
   \left( \begin{array}{ccc} m_1/t_1& m_2/t_2\\ m_3/t_3&  m_4/t_4 \end{array}\right)$, where $m_i/t_i \in Q_{n + 1}$.
   Suppose that $j \geq 1$ and   
    $\phi (\left( \begin{array}{ccc} 1/p_{n}^{j}& 0 \\ 0& 1/p_{n}^{j} \end{array}\right) )=
    \left( \begin{array}{ccc} m_{1,j}/t_{1,j} & m_{2,j}/t_{2,j}\\ m_{3,j}/t_{3,j}&  m_{4,j}/t_{4,j} \end{array}\right)$, where
    $p_n$ does not odd non of  $t_{i,j}$ for all $1 \leq i \leq 4$. 
    Then since $\phi$ is additive, we can easily see that 
    $\left( \begin{array}{ccc} m_{1,j} p_n^{j}/t_{1,j} & m_{2,j}p_n^{j}/t_{2,j}\\ m_{3,j}p_n^{j}/t_{3,j}&  m_{4,j}p_n^{j}/t_{4,j} \end{array}\right) = 
     \left( \begin{array}{ccc} m_1/t_1& m_2/t_2\\ m_3/t_3&  m_4/t_4 \end{array}\right)$
  and this implies that $p_n^{j} \vert m_i$ for all $j \geq 1$ and $0 \leq i \leq 4$ and so $m_i = 0$, a contradiction.
So $Q _R$ does not have couniserial dimension.
} 
\end{example}

\section{\hspace{-6mm}.  Some Applications.}

A  right  $R$-module $M$ which has a composition series is called a
module of {\it finite length.} A right $R$-module $M$ is of  finite length  if and only if $M$ is both
artinian and noetherian. The length of a composition series of $M_R$ is said to be the
length of $M_R$ and is denoted by length$(M)$. Clearly, by Corollary \ref{2.3}, a module of finite length
has   couniserial dimension. The next result shows a relation between      couniserial dimension
   of a finite length module $M$ and  length$(M)$.

  \begin{proposition}\label{semi1} Let $M$ be a right  $R$-module of finite length. Then the following statements hold: \\
    {\rm (1)} If $N$ is a submodule of $M$, then {\rm c.u.dim}$(M/N) \leq $ {\rm c.u.dim}$(M)$.\\
    {\rm (2)}   {\rm c.u.dim}$(M)  \leq $ {\rm length}$(M)$.
  \end{proposition}
\begin{proof}
   {\rm (1) The  proof is by induction on $n$, where length$(M) = n$.  The case $n = 1$  is clear.
    Now, let $ n > 1 $ and
 assume that the assertion is true for all
modules with length less than $n$. If $N$ is a non-zero submodule of $M$, then the length$(M/N) <  n$. Thus for every proper 
submodule
 $K/N$ of
$M/N$,    by induction, c.u.dim$(K/N) \leq $ c.u.dim$(K) < $ c.u.dim$(M)$. Now,  Remark \ref{1.3} implies that c.u.dim$(M/N)  \leq $ c.u.dim$(M)$. \\
(2)   The  proof is by induction on  length$(M) = n$. The case $n = 1$ is clear. Now if $n > 1$ and $K$ is
 a proper submodule of $M$,
then, by assumption, c.u.dim$(K) \leq $ length$(K) < $ length$(M)$. Thus by Remark \ref{1.3}, c.u.dim$(M) \leq $ length$(M)$.  $~\square$}
\end{proof}

  Recall that an $R$-module $M$ is called {\it co-semisimple} if every simple $R$-module is $M$-injective, 
or equivalently, Rad$(M/N) = 0$ for every submodule $N\leq M$ (See \cite[Theorem 23.1]{wis}).
The next proposition gives a condition as to when  a module of finite length is semisimple. It may be of
interest to state that for the finite length $\Bbb{Z}$-module $\Bbb{Z} _4$, $\oplus _{i = 1}^{\infty}\Bbb{Z} _4$ does not possess
couniserial dimension.
  
\begin{proposition}\label{Artinian} Let $M$ be a non-zero right $R$-module of finite length.
Then $M$ is a semisimple $R$-module if and only if 
 for every submodule $N$ of $M$ the right $R$-module $\oplus _{i = 1}^{\infty }M/N$ has
couniserial dimension.
\end{proposition}
\begin{proof} 
($\Rightarrow$) c.f. Remark \ref{semisimple eq}.\\
($\Leftarrow$) 
For every submodule $N$ of $M$ the right $R$-module $\oplus _{i = 1}^{\infty }M/N$ has
couniserial dimension. Clearly, this also holds for any factor module of $M$. 
We will  proof the result   by induction on the length$(M) =n $. The case $n = 1$ is clear.
 Now assume that $n > 1$ and the result is 
true for all modules of length less that $n$. 
Let  $K $ be  a non-zero submodule of $M$.
  Since length$(M/K) < n$, by   
    the inductive hypothesis, $M/K$ is semisimple. 
Therefore, for every non-zero submodule $K$ of $M$, Rad$(M/K) = 0$.
If Rad$(M) = 0$, then $M$ is co-semisimple.
Let  $S$ be a simple submodule of $M$.  Consider 
the exact sequence $0 \longrightarrow S \longrightarrow M \longrightarrow M/S \longrightarrow 0$ which  splits, because $M$ is co-semisimple. 
 Therefore, $M$ is semisimple. Next suppose that 
   Rad$(M) \neq 0$.    Let $S$ be a  simple submodule of $M$.  Because by the above
Rad$(M/S) = 0$, we obtain  Rad$(M) \leq S$. This implies Rad$(M) = S$ and so $M$ has only one simple submodule. 
 Thus  Rad$(M) =$ soc$(M) = S$ is a simple module. 
   Suppose that  $M$ is not semisimple.  
Let $N$ be a maximal submodule of $M$.  Then for every submodule  $K \leq N < M$,  
$\oplus _{i = 1}^{\infty }N/K $ is a submodule of $ \oplus _{i = 1}^{\infty }M/K$ and thus 
$\oplus _{i = 1}^{\infty }N/K$ has couniserial 
dimension. Since length$(N) < n$, we conclude that $N$ is semisimple. 
Thus $N = $ soc$(M) = $ Rad$(M)$ is a simple module  and so  $M$ is of length $2$.  \\
\indent Now consider the descending chain
$$ N \oplus (\oplus _{i = 2}^{\infty }M ) > N^{(2)}  \oplus (\oplus _{i = 3}^{\infty }M ) > ... $$ 
of submodules of $\oplus _{i = 1}^{\infty }M $.
Using Proposition \ref{2.2}, there exists $k \geq 1$ such that
$N^{(k)} \oplus (\oplus _{i = k + 1}^{\infty }M) \cong N^{(k + 1)} \oplus (\oplus _{i = k + 2}^{\infty }M) $. Since
$ N^{(k + 1)}$ is finitely generated, there exists $m \geq 0$ and an $R$-module $T$, such that
$ N^{(k)} \oplus M^{m} \cong N^{(k + 1)} \oplus T$.
$N$ is simple and so it has cancellation property and thus $M^{m} \cong N \oplus T$. This implies
Rad$(T)$ is semisimple of length $m$ and   length$($soc$(T)) =  m-1$,  a contradiction. $~\square$ 
\end{proof}

Recall that a ring $R$ is called {\it right  bounded} if every essential right  ideal contains a
two-sided ideal which is essential as a right  ideal. A ring $R$ is called right  {\it fully bunded}
if every prime factor ring is right bounded. A right noetherian right fully bounded ring
is commonly abbreviated as a right FBN ring. Clearly all commutative noetherian rings are example of right FBN rings. 
Finite matrix rings over commutative noetherian rings are a large class of right FBN rings  which are not commutative. 
In  \cite [Theorem  2.11] {Hiranocom}, Hirano and  et.al. showed  that a right FBN ring $R$  is semisimple if and only if every right 
 module of 
finite length is semisimple. As a consequence of the above proposition we have:

\begin{corollary}
A right FBN ring $R$ is semisimple
 if and only if for every  finite length module  $M$, the module $\oplus _{i = 1}^{\infty }M$
has couniserial dimension.  
\end{corollary}

  \begin{proposition} \label{anti is injective}
   Let $P$ be  an anti-coHopfian projective right $R$-module. If $\oplus_{ i = 1}^{\infty}E(P) $ has couniserial dimension, then
  $P$ is injective.
  \end{proposition}
  \begin{proof}
   We  first show that $P$ has cancellation property. 
 Let $M = P \oplus B \cong P \oplus B'$. So there exist  submodules $P'$ and $C $ of $M$ such that  $M = P \oplus B =  P' \oplus C$
 and
 $P' \cong P$ and $ C \cong B'$. If $p_1$ is a  projection map from $M = P \oplus B$ on to $P$. Then with
 restriction of $p_1$ to $C$  we have an
 exact
 sequence $ 0 \longrightarrow C \cap B \longrightarrow C \longrightarrow I \longrightarrow 0$, such that $I$ is a submodule of $P$.
 Note that every submodule of $P$  is  projective, because it is anti-coHopfian. So    $I$ is projective and
 thus  $C \cong C \cap B \oplus I$. Similarly   by considering map $p_2$ from $M = P' \oplus C$ to $P'$ we
 have
 $B \cong C \cap B \oplus J$ for some submodule $J$ of $P'$. Since $J \cong I \cong P$, we have $B\cong C$ and so $B\cong B'$.
Then $P$ has cancellation property. Now consider the
   descending
   chain
  $$ P \oplus ( \oplus _{ i = 2}^{\infty}  E(P)) \geq  P^{(2)} \oplus ( \oplus _{ i = 3}^{\infty} E(P)) \geq ... $$
   of submodules of
  $\oplus _{i = 1}^{\infty} E(P)$.
   Then, by Proposition \ref{2.2},
   there exists $n\geq 1$ such that
   $$P^{(n)} \oplus ( \oplus _{i = n + 1}^{\infty}E(P))\cong P^{(n + 1)}\oplus(\oplus _{i = n + 2}^{\infty} E(P))$$
    and so
    $\oplus _{i = n + 1}^{\infty} E(P) \cong P \oplus (\oplus _{i = n + 2}^{\infty} E(P))$ , because $P$ is cancelable . Since $P$ 
    is finitely
    generated, there exists a right module $L$ such that  for some $k$, $E(P) ^{k} \cong P\oplus L $. This shows $P$ is injective. 
    $~\square$

  \end{proof}

 As a consequence of the  above proposition we have the following corollary:

\begin{corollary} \label{domain}
Let  $R$ be  a principal right ideal domain with  maximal right quotient ring  $Q$ ( which is a division ring).
 If  the right $R$-module $\oplus_{i = 1}^{\infty} Q$  has couniserial dimension, then
  $R = Q$.
\end{corollary}

We need the following lemmas to prove the next  theorem. Using Proposition \ref{2.2} we can see that:
 
\begin{lemma} \label{factor}
Let $I$ be a two sided ideal of $R$ and $M$ be an $R/I$-module. If $M$ as $R$-module has couniserial dimension, then
$M$ as $R/I$-module  has couniserial dimension.
\end{lemma}

\begin{lemma} \label{notherian uniform}
If all finitely generated right modules have couniserial dimension, then every right module contains a noetherian uniform module. 
\end{lemma}
\begin{proof}
By Lemma \ref{anti-coHopfian} it is enough to show that every cyclic module contains an anti-coHopfian module.
Let $M$ be a  non-zero cyclic right module which does not contain anti-coHopfian module and let $S$ be a simple module.
$M$ is not anti-coHopfian, then $M$  has a
 non-zero submodule $M_1 \ncong M$ and $M_ 1$ has a non-zero submodule $M_2$ such that
$M_2 \ncong M_1$. By continuing in 
 this manner we have a descending chain $ S \oplus M \geq S\oplus  M_1 \geq S \oplus M_2 \geq ... $ of
submodules of
$S \oplus M$. Since $S\oplus M$ is finitely generated,   by Proposition \ref{2.2}, $ S \oplus M_n \cong  S \oplus M_{n + 1}$ 
for some $n$. This implies that
 $M_n \cong M_{n + 1}$ for some $n$, because $S$ is cancellable and this is a contradiction. $~\square$
\end{proof}

\begin{theorem} \label{final} For a ring $R$ the following are equivalent. \\
{\rm (1)} $R$ is a semisimple artinian ring.\\
{\rm (2)} All right $R$-modules have couniserial dimension.\\
{\rm (3)} All left $R$-modules have couniserial dimension.\\
{\rm (4)} All right $R$-modules have uniserial dimension.\\
{\rm (5)} All left $R$-modules have uniserial dimension.
\end{theorem}
\begin{proof}
For equivalence  of  (1), (4) and (5) refer  \cite [Theorem 2.6]{j.algebra}. \\
$(1) \Rightarrow (2)$. This  is clear by Corollary \ref{finite simple}.\\
$(2) \Rightarrow (1)$. At first we show $R$ satisfies ascending chain condition on two sided deals.
Let $I_1 \leq I_2 \leq ... $ be a chain of ideals of $R$. Since the right module $\oplus _{i = 1} ^{\infty} R/I_i $ has  couniserial
  dimension,
  there exists $n$ such that,  for each $j \geq n$,  $\oplus _{i = n} ^{\infty} R/I_i  \cong \oplus _{i = j} ^{\infty} R/I_i $.
  Thus they have the same
   annihilators and so  for each  $j \geq n$, $I_n = I_j$.
    Suppose  $R$ is  a non-semisimple  ring.
 By  Lemma \ref{factor} every module over a factor ring of $R$ also has couniserial
 dimension. Thus by invoking the ascending chain condition on two sided ideals we may assume
  $R$ is not semisimple artinian but every factor ring of $R$ is semisimple artinian.
Using Lemma \ref{semisimple1}, $R$ is a right V-ring. First   let us assume that  $R$ is  primitive.
 So, by  Theorem \ref{2.5}, $R$ is a  prime right  Goldie ring. 
  By \cite [Theorem 5.16] {simple noetherian ring},
 a prime right V-ring right Goldie is simple.  By Lemma \ref{notherian uniform},
  $R$  has a right noetheian uniform submodule and so  using  \cite [Corollary 7.25] {goodearl},
    $R$ is right noetherian.
     Now we show that  $R$ is Morita equivalent to
  a domain. By \cite [lemma 5.12]{5}, the  endomorphism ring of every uniform right  ideal of a prime right
   Goldie ring  is  a right ore domain.
   So by \cite [Theorem 1.2] {simple ring},
   it is enough to show that $R$ has a uniform projective generator $U$. Let us
    assume that $R$ is not uniform and u.dim$(R) = n$ and let $U$ be
   a  uniform right ideal of $R$. By  \cite [Corollary 7.25] {goodearl}, $U^{n}$ can be embedded in $R$ and also $R$ can be embedded in
    $U^{n}$. Then  c.u.dim$(R) = $ c.u.dim$(U^{n})$ and hence $R\cong U^{n}$, because $R$ is not uniform. Thus  $U$ is a projective generator uniform
     right ideal of $R$. So $R$ is Morita equivalent to a domain.  Now Lemma \ref{morita} 
      and  Lemma \ref {example} and Corollary  \ref{domain} show  that
     $R$ is simple artinian, a contradiction. So  $R$ is not primitive, but every primitive factor ring is artinian (indeed all proper factor 
     rings are artinian).
   Then since $R$ is a right V-ring,  by  \cite{Bacel proc}, $R$ is regular and $\Sigma$-V-ring.
 Also  every right ideal  contains a non-zero uniform right ideal, hence minimal.  So  $R$ has  non-zero essential soc$(R)$.
But   $R$ is $\Sigma$-V-ring and by Corollary \ref{finite simple} ,
 we have only finitely many non-isomorphic simple modules.  Thus  soc$(R)$ is injective.
This implies $R$ is  semisimple, a contradiction. This completes the proof. $~\square$
\end{proof}

{\bf ~~~~~~~~~~~~~~~~~~~~~~~~~~~~~~ Summary}

 This paper defines couniserial dimension of a module that measures how far a module is from being uniform. The results 
proved in the paper demonstrate  its importance for studding 
 the structure of modules and rings and is a beginning of a larger project to 
study its impact. We close with some open questions:\\
1) Does a module with arbitrary couniserial dimension possesses indecomposable dimension?\\
2) Is there a theory for modules with both finite uniserial and couniserial dimensions that parallels to Krull-Schmidt-Remak-Azumaya 
 theorem?

{\bf ~~~~~~~~~~~~~~~~~~~~~~~~~~~~~~Acknowledgments}

This paper was written   when the third author was visiting Ohio University, United States during May-August 2014. 
She wishes to express her deepest gratitude to Professor S. K. Jain for his kind guidance in her research project and
 Mrs. Parvesh  Jain  for the warm hospitality extended to her during her stay.
She would also like to express her thanks to Professor E. Zelmanov for his kind invitation to visit the  of University of
Californian at San Diego and to give a talk.


\begin{thebibliography}{99}

\bibitem {Anderson}  F. W. Anderson and, K. R. Fuller,   Rings and Categories of Modules, second ed., Grad. Texts in Math.,
Vol. 13, Springer,  Berlin, (1992).
\bibitem{2} G. Baccella, On flat factor rings and fully right idempotent rings, Ann. Univ. Ferrara Sez. 26 (1980), 125-141
\bibitem{Bacel proc} G. Baccella, Von Neumann regularity of V-rings with artinian primitive factor
rings, Proc. Amer. Math. Soc., 103, 3 (1988), 747-749.
\bibitem {simple noetherian ring}  J. Cozzens and C. Faith, Simple noetherian rings, Cambridge Univ. Press, Cambridge, (1975).
\bibitem {Uniserial with nil} N. I. Dubrovin, An example of a chain primitive ring with nilpotent elements,
 Mat. Sb. (N.S.) 120(162) (1983), no. 3, 441-447 (Russian). MR 691988 (84f:16012)
\bibitem{j.a.its} A. Ghorbani, Z. Nazemian, On commutative rings with uniserial dimension, 
Journal of Algebra and Its Applications,   14 (1)  (2015), 1550008. (Available on line).
\bibitem{5}   A. W. Goldie, Rings with maximum condition, Lecture Notes, Yale University, New Haven, Conn., (1961). 
\bibitem{goodearl} K. R. Goodearl, An Introduction to Non commutative noetherian Rings, 
London Math. Soc. 61, Cambridge University Press, Cambridge, (2004).
\bibitem{Gooderlnonsingular} K. R. Goodearl, Ring Theory: Nonsingular Rings and Modules, Dekker, New York (1976).
\bibitem {good} \textrm K. R. Goodearl, {\it Surjective endomorphisms of finitely generated modules,}
Comm. Algebra. 15 (1987), 589-609.
\bibitem{simple ring} R. Hart, J. C. Robson,
 Simple rings and rings Morita equivalent to ore domains. Proc. London Math. Soc.  21 (3) (1970),  232-242.

%\bibitem{8} Hirano. Y, Hong. C.H, Kim J.Y , Park.J.K,  On Injective Modules Whose Endomorphism
%Rings Are Simple artinian, Comm. Algebra, 27(1999), no.3, 1385-1391.

%\bibitem{lift} J. Clark, C. Lomp, N. Vanaja and R. Wisbauer, Lifting Modules. Supplements and
%Projectivity in Module Theory Frontiers in Math. Boston: Birkh\"{a}user, (2006).
%
%\bibitem{Eisenbud} D. Eisenbud and P. Griffith, {\it Serial rings,}  J. Algebra. 17 (1971), 389-400.
%
%\bibitem{facchini} A. Facchini and L. Salce, {\it Uniserial modules: sums and isomorphisms of
%subquotients,} Comm. Algebra. 18 (2) (1990), 499-517.
%
%\bibitem{Fuller}   K. R. Fuller, {\it On indecomposable injectives over artinian rings,} Pacific J. Math. 29 (1969), 115-135.
%
%
\bibitem{lance small} R. Gordon and L. W. Small, Piecewise domains, J. Algebra 23 (1972), 553-564.
\bibitem{Hirano} Y. Hirano  and I. Mogami,  Modules whose proper submodules are non-hopf kernels,
 Comm. Algebra. 15 (8) (1987), 1549-1567.
\bibitem {Hiranocom}
Y. Hirano , E. Poon  and H. Tsutsui,  A Generalization of Complete Reducibility, Comm. Algebra. 40 (6) (2012), 1901-1910.
%
%
%\bibitem {Lam}  T. Y. Lam,  Lectures on Modules and Rings, Grad. Texts in Math., Vol.
%189, Springer,  New York, (1999).
%
%\bibitem{Nakayama} T. Nakayama, {\it On Frobeniusian algebras II}, Ann. of Math. 42 (1941), 1-21.
%
%\bibitem{Quasi book} W. K. Nicholson  and  M. F. Yousif,  Quasi-Frobenius Rings. Cambridge Tracts in
%Mathematics, 158, Cambridge University Press, (2003).
%
%\bibitem{pun} \textrm {G. Puninski}, \emph {Some model theory over an exceptional uniserial
%ring and decompositions of serial modules }, J. Pure Appl. Algebra 163: 319-337.
%
%%
%%
%\bibitem{Warfield} \textrm  R. B. Warfield, {\it Serial rings and finitely presented modules}, J. Algebra. 37 (1975), 187-222.
%
%%%%%%%%%%%
\bibitem{cyclic} S. K. Jain, Ashish K. Srivastava and Askar A. Tuganbaev, Cyclic modules and the structure of rings,
Oxford Mathematical
Monographs, Oxford University Press, (2012). 
\bibitem {crash} T. Y. Lam, A crash course on stable range, cancellation, substitution, and exchange. J. Algebra Appl. 3 (2004),
301-343.
\bibitem {Lam}  T. Y. Lam,  Lectures on Modules and Rings, Grad. Texts in Math., Vol.
189, Springer,  New York, (1999).
\bibitem{robson} J.C. McConnell, J.C. Robson, Noncommutative noetherian Rings, Wiley, New York (1987)
\bibitem{7} G. O. Michler and O. E. Villamayor, On rings whose simple modules are injective, J. Algebra 25 (1973), 185-201.

%%%%%%%%%%%%\bib
%\bibitem {Baccella} G. Baccella, semiartinian v-rings and semiartinian von neumann regular rings.


\bibitem {j.algebra} Z. Nazemian, A. Ghorbani, M. Behboodi, Uniserial dimension of modules, J. Algebra, 399 (2014), 894-903.
%\bibitem {kurshan} R. P. Kurshan, Rings whose cyclic modules have finitely generated socles, J. Algebra, 15 (1970), 376 - 386.
%\bibitem{rib}  P.  Ribenboim,  Rings and modules.  Tracts in mathematics, 24, interscience Publishers, (1969).

\bibitem{stoll}  R. R. Stoll,  Set Theory and Logic. Dover Publication Inc, New York,  (1961).

\bibitem {varadarjan}   K.Varadarajan, Anti hopfian  and anti  coHopfian modules, 
AMS Contemporary  Math Series 456, (2008), 205-218.

\bibitem {wis}  R. Wisbauer, Foundations of Module and Ring Theory, Gordon and Breach, Reading, (1991).


\end{thebibliography}
\end{document}